\documentclass[11pt]{article}
\usepackage[utf8]{inputenc}
\usepackage{amsmath,amssymb,amsthm}
\usepackage{color}
\usepackage[linktocpage,hyperindex,colorlinks,citecolor=blue,linkcolor=blue,urlcolor=blue]{hyperref}

\newtheorem{theorem}{Theorem}
\newtheorem{lemma}{Lemma}
\newtheorem{proposition}{Proposition}
\theoremstyle{remark}
\newtheorem{remark}{\bf Remark}
\newtheorem{definition}{\bf Definition}

\newcommand{\R}{\mathbb R}
\newcommand{\N}{\mathbb N}

\newcommand{\ve}{\varepsilon}
\newcommand{\ds}{\displaystyle}
\newcommand{\al}{\alpha}
\newcommand{\la}{\lambda}
\newcommand{\lf}{\left}
\newcommand{\rg}{\right}
\newcommand{\lap}{\Delta}

\providecommand{\abs}[1]{\lvert#1\rvert}
\providecommand{\norm}[1]{\lVert#1\rVert}

\title{Nonradial solutions for the H\'enon equation close to the threshold}

\author{Pablo Figueroa \quad and \quad S\'ergio L. N. Neves}

\begin{document}

\maketitle
\begin{abstract}
\noindent We consider the H\'enon problem
\begin{equation*}
\begin{cases}
-\Delta u = \abs{x}^{\alpha} u^{\frac{N+2+2\alpha}{N-2}-\varepsilon} & \ \ \text{in} \ B_1, \\
u > 0 & \ \ \text{in} \ B_1, \\
u=0 & \ \ \text{on} \ \partial B_1, 
\end{cases}
\end{equation*}
where $B_1$ is the unit ball in $\R^N$ and $N\geqslant 3$. For $\varepsilon > 0$ small enough, we use $\alpha$ as a paramenter and prove the existence of a branch of nonradial solutions that bifurcates from the radial one when $\alpha$ is close to an even positive integer.

\end{abstract}

\emph{Keywords}: Bifurcation, H\'enon problem, Nonradial solutions
\\[0.1cm]

\emph{2010 AMS Subject Classification}: 35B32, 35J25, 35J60

\section{Introduction}

In this paper we consider the  H\'enon problem
\begin{equation}\label{eq1p}
\begin{cases}
-\Delta u = \abs{x}^{\alpha} u^p & \ \ \text{in} \ B_1, \\
u > 0 & \ \ \text{in} \ B_1, \\
u=0 & \ \ \text{on} \ \partial B_1. 
\end{cases}
\end{equation}
on the the unit ball $B_1 \subset \R^N$ with $N\geqslant 3$, $\al>0$ and $1<p<p_{\alpha} := \frac{N+2+2\alpha}{N-2}$.

The equation \eqref{eq1p} was introduced in \cite{H} by H\'enon in the study of stellar cluster in spherically symmetric setting and it is known as Hénon equation. We mention here some references but, since there is a vast literature regarding this problem and related ones, we remind that the list is far from complete. One of the earliest papers in this subject is \cite{Ni} by Ni,  where he proved the existence of a radial solution to \eqref{eq1p} for every $\alpha >0$ and $1<p<p_\al$ by using variational methods in the space of radial functions. Moreover, if $p\geqslant p_\al$, it follows from a Pohozaev type argument that \eqref{eq1p} admits no solutions, see for example \cite{Na}. Hence, the exponent $p_\al$ is the \emph{threshold} between existence and nonexistence. We avoid the use of the term \emph{critical} since this term is used to refer to the exponent $p_0 = \frac{N+2}{N-2}$ because of the critical Sobolev exponent.

The study of problem \eqref{eq1p} reveals some interesting phenomena. In particular, several questions arising naturally such as existence, multiplicity and qualitative properties of solutions have given the H\'enon equation an interesting role in nonlinear analysis and critical point theory. For instance, the existence result of Ni provides solutions for values $p$ above the critical Sobolev exponent. Another feature of this problem is that nonradial solutions might appear since the weight $|x|^{\alpha}$ is increasing and the symmetry result of Gidas, Ni and Nirenberg \cite{GNN} does not apply. Indeed, in \cite{SSW} Smets, Su and Willem proved the existence of nonradial solutions of \eqref{eq1p} in the subcritical case $1 < p < \frac{N+2}{N-2}$ by showing that the ground state solutions are not radial for $\alpha$ sufficiently large, roughly speaking they studied the asymptotic properties of the ground state levels in the spaces $H_{0,\text{rad} }^1(B_1)$ and $H_{0}^1(B_1)$. Working in some different subspaces of $H_{0}^1(B_1)$, Serra \cite{S} proved the existence of nonradial solutions for $p=\frac{N+2}{N-2}$, again for $\al$ sufficiently large. Similar ideas were used by Badiale and Serra \cite{BS} to obtain multiplicity results for some supercritical values of $p$ and for $\alpha$ large.

Another approach to prove the existence of nonradial solutions is to use perturbation methods and the well known Lyapunov-Schimdt finite dimensional reduction. In this setting, we refer to the papers of Peng \cite{P} and Pistoia and Serra \cite{PS}, where they used the exponent $p$ as a parameter and constructed solutions which blow up on points of the boundary as $p\to\frac{N+2}{N-2}$ from the left. In \cite{LiuP} the authors considered the case when $p\to\frac{N+2}{N-2}$ from the right. We cite also \cite{HCZ,WY} where the authors proved the existence of infinitely many solutions in the critical case $p=\frac{N+2}{N-2}$, in this case the parameter is the number of peaks of the solutions. Additionally, we cite the paper \cite{GG12} where the authors prove the existence of solutions to the Hénon equation in more general bounded domains, with $p$ close to $p_{\alpha}$ and using the Lyapunov-Schimdt reduction method.

It is also possible to find nonradial solutions via bifurcation methods. Amadori and Gladiali \cite{AG} proved the existence of at least one unbounded branch in the Hölder space $C^{1,\gamma}(\bar B_1)$ of nonradial solutions to \eqref{eq1p} that bifurcate from the radial one for some $\bar p\in (1,p_\al)$ with $\al\in(0,1]$ fixed. The used method does not allow to know if $\bar p$ is either subcritical or supercritical. 

So far the existence results of nonradial solutions of \eqref{eq1p} apply mostly for the case of $p$ subcritical and for some values of supercritical $p$. However, to our best knowledge, there are no results concerning the existence of nonradial solutions for supercritical values of $p$ close to $p_\al$. In this paper, we want to study the case of $p$ close to the threshold $p_{\alpha}$. More specifically, our aim is to prove the existence of nonradial solutions for the Hénon equation in the following case
\begin{equation} \label{eq1}
\begin{cases}
-\Delta u = \abs{x}^{\alpha} u^{p_\al-\varepsilon} & \ \ \text{in} \ B_1, \\
u > 0 & \ \ \text{in} \ B_1, \\
u=0 & \ \ \text{on} \ \partial B_1, 
\end{cases}
\end{equation}
for $\varepsilon > 0$ small and $\alpha>0$. In order to state our results, let us introduce some elements. Let $u_{\varepsilon,\alpha}$ be the unique radial solution to \eqref{eq1}, see Proposition \ref{prop1} for a proof. 

\begin{definition}
We say that a nonradial bifurcation occurs at $(\bar\al,u_{\ve,\bar\al})$ if in every neighborhood of $(\bar\al,u_{\ve,\bar\al})$  in $(0,+\infty)\times C^{1,\gamma}_0(B_1)$ there exists a point $(\al,v_\al)$ with $v_\al$ a nonradial solution to \eqref{eq1}.
\end{definition}  

Our main result is the following.

\begin{theorem} \label{t1}
Let $\alpha_k$ be an even positive integer. Then, given $\rho > 0$, there exists $\varepsilon_0 > 0$ such that for every $\varepsilon \in (0,\varepsilon_0)$ there exists $\alpha_k^{\varepsilon} \in (\alpha_k-\rho,\alpha_k+\rho)$ and a continuum of nonradial solutions of \eqref{eq1} bifurcating from the pair $\big(\alpha_k^{\varepsilon},u_{\varepsilon,\alpha_k^{\varepsilon}} \big)$. Moreover, it holds $\al_k^\ve \to \alpha_k$ as $\ve \to 0$.
\end{theorem}

The result is inspired by \cite{GGN}, where the authors considered the H\'enon equation in the whole space with $p=p_{\alpha}$, namely,

\begin{equation}\label{problemrn}
\begin{cases}
-\Delta u =C_{N,\al} |x|^{\alpha}u^{p_\al }, \ \ & \text{in} \ \R^N \\
u > 0, & \text{in} \ \R^N,
\end{cases}
\end{equation}
where $N\ge3$ and $C_{N,\al}:=(N-2)(N+\al)$, which has the explicit family of radial solutions
\begin{equation}\label{radsolrn}
 U_{\lambda,\alpha}(x) = \frac{\lambda^{\frac{N-2}{2}}}{(1+\lambda^{2+\alpha}|x|^{2+\alpha})^{\frac{N-2}{2+\alpha}}},\ \ \text{with }\ \la>0.
 \end{equation}
They have characterized the solutions of the linearized equation of \eqref{problemrn} around the radial solution $U_{1,\alpha}$, provided a formula for the Morse index of these solutions and deduced that the Morse index change as the parameter $\alpha$ crosses the even integers. After that it was proved the existence of nonradial solutions of \eqref{problemrn} bifurcating from $(\alpha,U_{1, \alpha})$ for $\alpha$ even, by studying first an approximate problem on a ball with radius $\frac{1}{\ve}$ and passing to the limit as $\ve\to 0$ with a careful analysis of several estimates. We want to apply some of these ideas to our problem since, after a proper rescaling, \eqref{eq1} can be considered as an approximating problem for \eqref{problemrn} in an expanding ball as $\varepsilon \to 0$. 

Finally, we point out some comments about the proof of the theorem. We consider the curve of radial solutions of \eqref{eq1} and  $\alpha$ as a parameter, if $\varepsilon$ is small enough there is a change in the Morse index of the radial solution of \eqref{eq1} for $\alpha$ close to an even integer and so we can apply the classical bifurcation theory to deduce the existence of a branch of nonradial solutions to the rescaled problem. We deduce Theorem \ref{t1} by using the bifurcation result for the rescaled problem, which in some sense resembles the approximate problem in \cite{GGN}. We also make use of the nondegeneracy of $u_{\ve,\al}$ in the space of radial functions and the uniform convergence of the eigenvalues and eigenfunctions. An important ingredient is that the only eigenvalue that plays a role is the first one, see section \ref{linprob} for more details.
\begin{remark}
The nonradial solutions given by Theorem \ref{t1} belong to the subspace of functions invariant with respect to the subgroup $O(N-1)$ in $C^{1,\gamma}_0(B_1)$, where $O(h)$ is the orthogonal group in $\R^h$. With the same technique one can also prove a bifurcation result in different subspaces as long as there is an odd change in the Morse index of the radial solution $u_{\varepsilon,\al}$. For example, if $\alpha_k = 2(k-1)$ with $k$ even then it is possible to find $[\frac{N}{2}]$ different continua of positive nonradial solutions of \eqref{eq1} bifurcating from $(\alpha_k^{\varepsilon},u_{\varepsilon,\alpha_k^ {\varepsilon}})$ as in \cite[Theorem 3.8]{GGN}. Precisely, every different continua belongs to $O(h)\times O(N-h)$ for $1\leqslant h\leqslant [\frac{N}{2}]$.
\end{remark}

\begin{remark}
There is also the question about what happens with the continua of nonradial solutions of \eqref{eq1} in the limit case $\varepsilon \to 0$. After a suitable rescaling, they might converge to a branch of nonradial solutions of the limit problem \eqref{problemrn}. This was the subject of \cite{GGN} where the authors proved the existence of nonradial solutions of \eqref{problemrn}, but in \cite{GGN} it was used a different approximating problem. 
\end{remark}

The paper is organized as follows. In section \ref{sec1}, some properties of the radial solutions $u_{\ve,\al}$ are shown in order to use them in next sections. In section \ref{resprob} we present the rescaled problem and prove the convergence to \eqref{problemrn}. The linearized equation of the rescaled problem is studied in section \ref{linprob} in order to state the convergence of eigenvalues and eigenfunctions, we also deduce the existence of \linebreak $\al_k^\ve\sim 2(k-1)$ where the Morse index changes. Finally the last section \ref{proofresult} is devoted to the proof of Theorem \ref{t1}.

\section{Properties of the radial solutions}\label{sec1}

In this section, we present some important facts about the radial solutions $u_{\ve,\al}$ to \eqref{eq1} for $\ve,\al>0$ such as the existence, uniqueness, asymptotic behavior of its $L^\infty$-norm $\|u_{\ve,\al}\|_{L^\infty(B_1)}=u_{\ve,\al}(0)$, an important estimate to pass to the limit $\ve\to 0$ and the nondegeneracy in the space of radial functions. It is worth to mention that some of these properties are uniform for $\al$ in compact subsets of $(0,+\infty)$. 

\begin{proposition}\label{prop1}
Suppose that $1 < p < \frac{N+2+2\alpha}{N-2}$, then the equation \eqref{eq1p}
admits a unique radial solution.
\end{proposition}
\begin{proof} This result is well known but we write the proof here for the sake of completeness.
Consider the o.d.e.
$$ u'' + \frac{N-1}{r}u' + r^{\alpha}u^p = 0.$$
If $u$ is a solution to the above equation then $u_{\lambda}(x) = \lambda^{\frac{2+ \alpha}{p-1}}u(\lambda x)$ is also a solution. It is known that the i.v.p.
\begin{equation*}
\begin{cases}
u'' + \dfrac{N-1}{r}u' + r^{\alpha}u^p = 0 \\
u(0) = a > 0 \qquad u'(0)=0
\end{cases}
\end{equation*}
admits a unique solution, see for instance \cite[Prop. 2.35]{NN}. If $u$ and $v$ are two solutions of \eqref{eq1p} then, taking $\lambda^{\frac{2+\alpha}{p-1}} = u(0)/v(0)$, by the uniqueness it follows that $u(r) = \lambda^{\frac{2+\alpha}{p-1}}v(\lambda r)$. Since $u(1)=0=v(1)$, then we must have $\lambda =1$. Therefore, $u=v$ in $B_1$.
\end{proof}

\noindent Another simple way to prove the above uniqueness result is to make a change of variables and reduce to the classical case $\alpha=0$, see \cite[Theorem A.2]{GGN}.

Now, we study the asymptotic behavior of $u_{\ve,\al}$ in $L^\infty(B_1)$. Notice that, since $u_{\ve,\al}'(r)<0$ for $r>0$, it follows that $\|u_{\ve,\al}\|_{L^\infty(B_1)}=u_{\ve,\al}(0)$. The first part of the next result was essentially proved in \cite{LP}. Here, we give a simpler proof and extend the convergence uniformly for $\al$ in compact subsets of $[0,+\infty)$.

\begin{lemma}\label{lem1}
It holds 
\begin{equation} \label{eq7} 
\lim_{\varepsilon \to 0} \varepsilon u_{\varepsilon,\alpha}^2(0) = \frac{2(2+\alpha)}{N-2}[(N-2)(N+\alpha)]^{\frac{N-2}{2+\alpha}} \dfrac{\Gamma\left( \frac{2(N+\alpha)}{2+\alpha} \right)}{\left[\Gamma\left(\frac{N+\alpha}{2+\alpha}\right)\right]^2}:=M(N,\al)
\end{equation}
uniformly for $\alpha$ in compact subsets of $[0,\infty)$.
In particular, defining \linebreak $\mu_{\varepsilon,\alpha} = \left( u_{\varepsilon,\alpha}(0)\right)^{-2}$ it follows that
\begin{equation*}
\mu_{\varepsilon,\alpha}^{\varepsilon} \to 1 \quad \text{as} \quad \varepsilon \to 0
\end{equation*} 
uniformly for $\alpha$ in compact subsets of $[0,\infty)$.
\end{lemma}
\begin{proof}
Consider the function defined by 
$$v_{\varepsilon,\alpha}(r) = \left[ \frac{2}{2+\alpha} \right]^{\frac{2}{p_{\alpha}-1-\varepsilon}} u_{\varepsilon,\alpha}(r^{\frac{2}{2+\alpha}}),$$
then $v_{\varepsilon,\alpha}$ satisfies the equation
\begin{equation} \label{eq5}
\begin{cases}
v'' + \dfrac{m-1}{r}v' + v^{\frac{m+2}{m-2} -\varepsilon} = 0, \qquad r \in (0,1) \\
v'(0) = 0 =v(1)
\end{cases}
\end{equation}
where $m=\frac{2(N+\alpha)}{2+\alpha}$.

Now, we apply the result in \cite[Theorem A]{AP2} and get
\begin{align}
\lim_{\varepsilon \to 0} \varepsilon v_{\varepsilon,\alpha}^2(0) &= \frac{4}{m-2}[m(m-2)]^{\frac{m-2}{2}}\frac{\Gamma(m)}{\left[ \Gamma\left( \frac{m}{2}\right)\right]^2} \notag \\
\lim_{\varepsilon \to 0} \left[ \left(\frac{2}{2+\alpha}\right)^{\frac{4}{p_{\alpha}-1-\varepsilon}} \varepsilon \, u_{\varepsilon,\alpha}^2(0)\right] &= \frac{4}{m-2}[m(m-2)]^{\frac{m-2}{2}}\frac{\Gamma(m)}{\left[ \Gamma\left( \frac{m}{2}\right)\right]^2}. \label{eq6}
\end{align}
Actually, the results in \cite{AP2} were proved for the o.d.e \eqref{eq5} in the case $m \in \N$, but we readily see that the same results apply for non integer $m$ as long as $m > 2$.

By \eqref{eq6}, recalling that $m=\frac{2(N+\alpha)}{2+\alpha}>2$, we get
$$ \lim_{\varepsilon \to 0}  \varepsilon \, u_{\varepsilon,\alpha}^2(0)= \frac{2(2+\alpha)}{N-2}\left[(N+\alpha)(N-2)\right]^{\frac{N-2}{2+\alpha}}\frac{\Gamma\left(\frac{2(N+\alpha)}{2+\alpha}\right)}{\left[ \Gamma\left( \frac{N+\alpha}{2+\alpha}\right)\right]^2}.$$

Finally, by a careful analysis of \cite{AP2}, one can check that the convergence \eqref{eq7} is indeed uniform for $\alpha$ in compact subsets of $[0,\infty)$.
\end{proof}

An important estimate for $u_{\ve,\al}$ is proved in the following result, that allows to study the rescaled problem and pass to the limit $\ve\to0$.

\begin{lemma}\label{lem2}
The following estimate holds
$$u_{\varepsilon,\alpha}(x) \leqslant \left( \frac{\mu_{\varepsilon,\alpha}^{(p_\alpha-1-2\varepsilon)/4}}{\mu_{\varepsilon,\alpha}^{(p_\alpha-1-\varepsilon)/2} + C_{N,\alpha}^{-1}\abs{x}^{2+\alpha}} \right)^{\frac{N-2}{2+\alpha}} \qquad \text{for all} \ \ x \in B_1.$$
\end{lemma}
\begin{proof}
Proceeding as in \cite{AP1} (see also \cite{AP2}), by using the transformation
$$t=(m-2)^{m-2}r^{2-m},\quad y(t)=v(r)$$
on \eqref{eq5} and Lemma 1 (iii) in \cite{AP1} (see also \cite[Lemma 1]{AP2}), we find that
\begin{equation}\label{dap}
y(t)<\gamma\left(1+\frac{1}{k-1}\frac{\gamma^{p-1}}{t^{k-2}}\right)^{-1/(k-2)},
\end{equation}
where $k=2\,\dfrac{m-1}{m-2}$, $p=\dfrac{m+2}{m-2}-\ve$ and
$$\gamma=\lim_{t\to+\infty}y(t)=v(0)=\left[ \frac{2}{2+\alpha} \right]^{\frac{2}{p_{\alpha}-1-\varepsilon}} u_{\varepsilon,\alpha}(0).$$
Replacing in \eqref{dap} and simplifying we get that for $r\in(0,1)$
\begin{equation*}
\begin{split}
u_{\ve,\alpha}(r^{\frac{2}{2+\alpha}})&<u_{\ve,\alpha}(0)\left(1+{m-2\over m}\,{[v(0)]^{\frac{4}{m-2}-\ve}\over (m-2)^2r^{-2}} \right)^{(2-m)/2}\\
&\leq u_{\ve,\alpha}(0)\left(1+{[u_{\ve,\alpha}(0)]^{p_\alpha-1-\ve}\over (N-2)(N+\alpha)}\,r^2 \right)^{-\frac{N-2}{2+\alpha}}
\end{split}
\end{equation*}
Hence, taking into account that $r=|x|$, the conclusion follows.
\end{proof}

We end this section by proving a result about the linearized problem of \eqref{eq1}  at $u_{\ve,\al}$, i.e.
\begin{equation} \label{eqv}
\begin{cases}
-\Delta v = (p_\alpha-\ve)\abs{x}^{\alpha} u_{\ve,\alpha}^{p_\alpha-1-\varepsilon}v, & \ \ \text{in} \ B_1 \\
\ \ v=0 & \ \ \text{on} \ \partial B_1 
\end{cases}
\end{equation}
which is important in the classical bifurcation theory.

\begin{lemma}\label{lem4}
The solution $u_{\varepsilon,\alpha}$ is \emph{nondegenerate} in the space of the radial functions, namely, if $v$ is a radial solution to  \eqref{eqv} then $v\equiv 0$.
\end{lemma}

\begin{proof}
Let $v$ be a radial solution to \eqref{eqv}. Similarly to \cite[Theorem 4.2]{DGP}, multiplying \eqref{eq1} by $v$ and \eqref{eqv} by $u_{\ve,\alpha}$, and integrating on $B_1$ we find that
\begin{equation}\label{ig0}
\int_{B_1}|x|^\alpha u_{\ve,\alpha}^{p_\alpha-\ve}v=0.
\end{equation}

On the other hand, direct computations show that the function given by \linebreak $\zeta(x)=x\cdot\nabla u_{\ve,\alpha}(x)$, $x\in B_1$, satisfies
$$-\Delta\zeta=-x\cdot\nabla(\Delta u_{\ve,\alpha})-2\Delta u_{\ve,\alpha}=(p_\alpha-\ve)|x|^\alpha u_{\ve,\alpha}^{p_\alpha-1-\ve}\zeta+(2+\alpha)|x|^\alpha u_{\ve,\alpha}^{p_\alpha-\ve}.$$

Now, multiplying the previous equation by $v$, the equation \eqref{eqv} by $\zeta$, integrating over $B_1$ and using \eqref{ig0} we find that
$$\int_{\partial B_1} \zeta {\partial v\over\partial \nu}\,d\sigma=\int_{\partial B_1} {\partial u_{\ve,\alpha}\over\partial \nu}\,{\partial v\over\partial \nu}\,d\sigma=0.$$
Since $u_{\ve,\alpha}$ and $v$ are radial functions, namely, $u_{\ve,\alpha}(x)=u_{\ve,\alpha}(|x|)$ and $v(x)=v(|x|)$, we get that $\nabla u_{\ve,\alpha}(x)=u_{\ve,\alpha}'(|x|)\dfrac{x}{|x|}$, so that $\dfrac{\partial u_{\ve,\alpha}}{\partial \nu}(x)=u_{\ve,\alpha}'(1)$ and similarly, $\dfrac{\partial v}{\partial \nu}(x)=v'(1)$ for $x\in\partial B_1$. Thus, $u_{\ve,\alpha}'(1)v'(1)=0$, but $u_{\ve,\alpha}'(1)<0$ by the Hopf's Lemma. Therefore $v'(1)=0$ and by continuation $v\equiv0$ since $v=v(r)$ satifies
\begin{equation*}
\begin{cases}
v'' + \dfrac{N-1}{r}v' +(p_\alpha-\ve) r^{\alpha}u_{\ve,\alpha}^{p_\alpha-1-\ve} v = 0,&\text{in } (0,1) \\
\ \ v(1) = 0= v'(1).
\end{cases}
\end{equation*}
This completes the proof.
\end{proof}

\section{The rescaled problem}\label{resprob}

Here, we present the rescaled problem mentioned in the introduction on which we deduce a bifurcation result by studying the change of the Morse index of the radial solution for $\al$ close to an even integer.
 
Now, denoting $\rho_{\ve}= \varepsilon^{-\frac{1}{N-2}}$ define the function $w_{\ve,\alpha}$ by
\begin{equation} \label{trans}
w_{\varepsilon,\alpha}(x) = \kappa_{\varepsilon,\alpha} u_{\varepsilon,\alpha}\left( \rho_{\ve}^{-1} x \right),\qquad\text{for } \rho_{\ve}^{-1} x\in B_1
\end{equation}
where $\kappa_{\varepsilon,\alpha}$ is defined by the relation $\ds\kappa_{\varepsilon,\alpha}^{1-(p_\alpha - \varepsilon)} =  C_{N,\alpha}\ \varepsilon^{-\frac{2+\alpha}{N-2}}.$ Notice that $w_{\varepsilon,\alpha}(x)$ satifies
\begin{equation} \label{eq2}
\begin{cases}
-\Delta w = C_{N,\alpha}\abs{x}^{\alpha} w^{\frac{N+2+2\alpha}{N-2}-\varepsilon}, & \ \ \text{in} \ B_{\rho_{\varepsilon}}\\
w > 0 & \ \ \text{in} \ B_{\rho_{\varepsilon}}\\
w=0 & \ \ \text{on} \ \partial B_{\rho_{\varepsilon}}
\end{cases}
\end{equation}
where $\ds B_{\rho_{\varepsilon}} = \left\{ x \in \R^N \ \ \vert  \ \ \ \abs{x} < \rho_{\ve}\right\}.$ Consider $u_{\varepsilon,\alpha}$ extended by zero outside of $B_1$, so that, $w_{\ve,\alpha}$ is also defined (as zero in $\R^N\setminus B_{\rho_{\ve}}$) in the whole $\R^N$.

Recall that $U_{\la,\al}$, with $\la>0$ in \eqref{radsolrn}, are the unique radial solutions to \eqref{problemrn} and notice that $U_{\la,\al}(0)=\la^{N-2\over 2}$.

As a consequence of Lemma \ref{lem1} and Lemma \ref{lem2} we have the following fact.
\begin{lemma} \label{lem3}
Let $K \subset (0,\infty)$ be a compact interval. Then, $w_{\varepsilon,\alpha}\longrightarrow U_{\alpha}$ as $\varepsilon \to 0$ uniformly for $(\alpha,x) \in K \times \R^N $, where
$$U_{\alpha}(x) = \frac{\lambda^{\frac{N-2}{2}}}{\left( 1+\lambda^{2+\alpha}\abs{x}^{2+\alpha}\right)^{\frac{N-2}{2+\alpha}}}, \quad\text{with} \quad \lambda^{\frac{N-2}{2}}: =[C_{N,\alpha}]^{-\frac{N-2}{2(2+\alpha)}} [M(N,\alpha)]^{\frac{1}{2}}$$
and $M(N,\alpha)$ is just the RHS of \eqref{eq7} in Lemma \ref{lem1} .
\end{lemma}


\begin{proof}
We shall use some ideas as in \cite[Lemma 3.6]{BP}. First, we consider $w_{\ve,\alpha}$ defined in $\R^N$ extended by zero outside of $B_{\rho_{\ve}}$. From the definition of $w_{\ve,\alpha}$ and Lemmas \ref{lem1} and \ref{lem2} it follows that 
\begin{align}
w_{\varepsilon,\alpha}(x) &\leqslant \kappa_{\ve,\alpha}\left( \frac{\mu_{\varepsilon,\alpha}^{\frac{2+\alpha}{2(N-2)}-\frac{\ve}{2}}}{\mu_{\varepsilon,\alpha}^{{2+\alpha\over N-2}-{\ve\over2}} + C_{N,\alpha}^{-1}\varepsilon^{2+\alpha\over N-2}\abs{x}^{2+\alpha}} \right)^{\frac{N-2}{2+\alpha}} \notag \\
&\leqslant \varepsilon^{-\frac 1 2} \kappa_{\ve,\alpha}\left( \frac{(\varepsilon^{-1}\mu_{\varepsilon,\alpha})^{\frac{2+\alpha}{2(N-2)}}}{(\varepsilon^{-1}\mu_{\varepsilon,\alpha})^{\frac{2+\alpha}{N-2}} + C_{N,\alpha}^{-1} \mu_{\varepsilon,\alpha}^{\frac{\varepsilon}{2}} \abs{x}^{2+\alpha}} \right)^{\frac{N-2}{2+\alpha}} \notag \\
&\leqslant \frac{C_N}{\Big(1 +\abs{x}^{2+\alpha}\Big)^{\frac{N-2}{2+\alpha}}},\label{dw}
\end{align}
for some constant $C_N$ uniform in $\varepsilon \in (0,\varepsilon_0)$, $\alpha \in K$ and $x \in \R^N$, in view of 
$$\ve\mu_{\ve,\al}^{-1}=M(N,\al)+o(1), \quad \mu_{\ve,\al}^{\ve\over 2}=1+o(1)\quad \text{and}\quad \varepsilon^{-\frac 1 2} \kappa_{\ve,\alpha}=C_{N,\al}^{-{N-2\over 2(2+\al)}} +o(1)$$
as $\ve\to 0$. Hence, the family $\{w_{\ve,\alpha}\}_\ve$ is bounded uniformly for $\alpha \in K$ in $L^{\infty}(\R^N)$. The elliptic regularity theory implies that $\{w_{\ve,\alpha}\}_\ve$ is equicontinuous on every
compact subset of $\R^N$. Hence, by the Arz\'ela-Ascoli Theorem, for every sequence $\varepsilon_n \to 0$ there exists a
subsequence, still denoted by $\{w_{\ve_n,\alpha}\}$, which converges to some function $\bar w$
uniformly on compact subsets of $\R^N$. Since $w_{\ve,\al}$ and $U_\al$ have uniform decay \eqref{dw}, the convergence is uniform in all $\R^N$. Taking the limit in \eqref{eq2}, using
Lemma \ref{lem1} again, we conclude that $\bar w$ is a radial positive function, i.e. $\bar w(x)=w(|x|)$, and satisfies
\begin{equation*}
\begin{cases}
-\Delta w = C_{N,\alpha}\abs{x}^{\alpha} w^{\frac{N+2+2\alpha}{N-2}},\ \ w>0, & \  \text{in} \ \R^N\\
\ \ w(0)=[C_{N,\alpha}]^{-\frac{N-2}{2(2+\alpha)}} [M(N,\alpha)]^{\frac{1}{2}},
\end{cases}
\end{equation*}
in view of  
$$w_{\varepsilon,\alpha}(0) = \kappa_{\varepsilon,\alpha} u_{\varepsilon,\alpha}\left(0\right)=C_{N,\al}^{-{N-2\over 2(2+\al)-\ve(N-2)} } \ve^{{2+\al\over 2(2+\al)-\ve(N-2)}-{1\over 2} } (\ve \mu_{\ve,\al}^{-1})^{1/2}$$
and 
$$ \lim_{\varepsilon \to 0^+}w_{\varepsilon,\alpha}(0)=[C_{N,\alpha}]^{-\frac{N-2}{2(2+\alpha)}} [M(N,\alpha)]^{\frac{1}{2}}=\lambda^{\frac{N-2}{2}}.$$
 Indeed, if $\bar w(r_0)=0$ for some $r_0>0$, then $\bar w$ is a solution to the problem
\begin{equation*}
\begin{cases}
-\Delta w = C_{N,\alpha}\abs{x}^{\alpha} w^{\frac{N+2+2\alpha}{N-2}}, & \ \ \text{in} \ B_{r_0}\\
w > 0 & \ \ \text{in} \ B_{r_0}\\
w=0 & \ \ \text{on} \ \partial B_{r_0}
\end{cases}
\end{equation*}
where $B_{r_0} = \left\{ x \in \R^N \ \vert   \ \abs{x} < r_0\right\}$, but it is known that the latter problem has no solution by the Pohozaev identity. Therefore, $\bar w>0$ in $\R^N$. Thus, by using \eqref{radsolrn} we conclude that $\bar w=U_\alpha$.

Finally, the convergence is uniform for $\alpha \in K$, because otherwise there exist $\delta_0>0$ and sequences $\ve_n>0$, $\ve_n\to0$ and $\al_n\in K$ such that for all $n\in\N$ it holds
\begin{equation}\label{nucal}
\lf\|w_{\ve_n,\al_n}-U_{\al_n}\rg\|_{L^\infty(\R^N)}\ge\delta_0.
\end{equation}
Hence, there exist subsequences still denoted by $\ve_n>0$ and $\al_n\in K$ such that $\ve_n\to0$ and $\al_n\to\al^*$ for some $\al^*\in K$. Arguing exactly as before, by using \eqref{dw}, $\{w_{\ve_n,\alpha_n}\}_n$ is bounded uniformly in $L^{\infty}(\R^N)$, consequently up to a subsequence converges uniformly in all $\R^N$ to a function $w^*$, and using the differential equation satisfied by $w_{\ve_n,\al_n}$, it follows that $w^*=U_{\al^*}$, which is a contradiction with \eqref{nucal} in view of the uniform convergence $U_{\al_n}\to U_{\al^*}$ in $\R^N$, this completes the proof.
\end{proof}

\begin{remark}
Taking into account the relation between $u_{\ve,\al}$ and $w_{\ve,\al}$ and Lemma \ref{lem4}, where we proved that $u_{\ve,\al}$ is radially nondegenerate in $B_1$, it readily follows that $w_{\ve,\al}$ is radially nondegenerate in $B_{\rho_\ve}$.
\end{remark}

\section{The linearized problem}\label{linprob}

In this section we consider the linearized operator of the rescaled problem \eqref{eq2} at $w_{\ve,\al}$, i.e.
\begin{equation} \label{plw}
\begin{cases}
-\Delta v = C_{N,\alpha} (p_{\alpha} - \varepsilon) \abs{x}^{\alpha} w_{\ve,\alpha}^{p_{\alpha}-\varepsilon -1}v, & \ \ \text{in} \ B_{\rho_{\varepsilon}} \\
v=0, & \ \ \text{on} \ \partial B_{\rho_{\varepsilon}}.
\end{cases}
\end{equation}
Following the ideas presented in \cite{GGN} (see also \cite{AG}), we decompose \eqref{plw} in radial part and angular part using the spherical harmonic functions, namely, we look for solutions of the form
$$v(r,\theta)=\sum_{k=0}^{+\infty}v_k(r)Y_k(\theta),$$
where $r=\abs x \in (0,\rho_{\ve})$, $\ds\theta={x\over |x|}\in \mathbb{S}^{N-1}$ and $\ds v_k(r)=\int_{\mathbb{S}^{N-1}}v(r,\theta)Y_k(\theta)\,d\theta$. Here, $Y_k=Y_k(\theta)$ denotes $k$-th spherical harmonic function, i.e. it satisfies
\begin{equation} \label{autflapb}
-\lap_{\mathbb{S}^{N-1}}Y_k=\sigma_kY_k\quad\text{in}\ \mathbb{S}^{N-1}
\end{equation}
where $\lap_{\mathbb{S}^{N-1}}$ is the Laplace-Beltrami operator on $\mathbb{S}^{N-1}$ with the standard metric and $\sigma_k$ is the $k$-th eigenvalue of $-\lap_{\mathbb{S}^{N-1}}$. It is known that 
\begin{equation} \label{autvlapb}
\sigma_k=k(N+k-2),\ \ k=0,1,2,\dots
\end{equation}
whose multiplicity is 
$$\dfrac{(N+2k-2)(N+k-3)!}{(N-2)!k!}$$ 
and that Ker$(\lap_{\mathbb{S}^{N-1}}+\sigma_k)=\mathbb{Y}_k(\R^N)|_{\mathbb{S}^{N-1}}$, where $\mathbb{Y}_k(\R^N)$ is the space of all homogeneous harmonic polynomials of degree $k$ in $\R^N$. Thus, $v$ satisfies \eqref{plw} if and only if $v_k$ is a solution of
\begin{equation}\label{eqvk}
\begin{cases}
\ds -v_k'' - \dfrac{N-1}{r}v_k'+{\sigma_k\over r^2}v_k=(p_\alpha-\ve) C_{N,\alpha}r^\alpha w_{\ve,\alpha}^{p_\alpha-1-\ve}v_k\quad \text{in}\ (0,\rho_{\ve}) \\
v_0'(0) = v_0(\rho_{\ve})=0,\quad \text{and} \quad v_k(0) = v_k(\rho_{\ve})=0,\quad \text{if} \ k>0.
\end{cases}
\end{equation}
In this way, we have that all the eigenfunctions of \eqref{plw} are given by $v_k(r)Y_k(\theta)$, if $v_k$ is a solution of \eqref{eqvk}. Notice that the function $\tilde v(r)= v_0(\rho_{\ve} r)$ is a radial solution to \eqref{eqv}. By Lemma \ref{lem4} we obtain that $\tilde v\equiv0$ (in other words, $w_{\ve,\al}$ is radially nondegenerate) and we get that \eqref{eqvk} admits non trivial solutions only if $k>0$. Hence, let us introduce the following eigenvalue problem,
\begin{equation}\label{eqz}
\begin{cases}
\ds -z'' - \dfrac{N-1}{r}z'-(p_\alpha-\ve) C_{N,\alpha}r^\alpha w_{\ve,\alpha}^{p_\alpha-1-\ve}(r) z=\Lambda {z\over r^2}\quad \text{in}\ (0,\rho_{\ve}) \\
z(0) = z(\rho_{\ve})=0,
\end{cases}
\end{equation}
which admits an increasing sequence of eigenvalues $\Lambda_{j}^{\varepsilon}(\alpha)$, $j\in\N$, which are simple, see for instance \cite{Z}. Thus, we obtain that \eqref{eqvk} is equivalent to find $\al>0$ and integers $j,k\ge 1$ such that 
\begin{equation}\label{eqmuk}
-\sigma_k=\Lambda_{j}^{\varepsilon}(\al),
\end{equation}
where $\Lambda_{j}^{\varepsilon}(\alpha)$ is an eigenvalue of \eqref{eqz}. Notice that, by Lemma \ref{lem3}, as $\ve\to 0$
\begin{equation} \label{erem5}
w_{\ve,\al}^{p_\alpha - \varepsilon - 1}(r)\to {\la^{2+\al} \over (1+\la^{2+\al}r^{2+\al})^2}
\end{equation}
uniformly for $r\in \R_+$ and for $\al$ in compact subset of $(0,+\infty)$, where $\la$ is defined in Lemma \ref{lem3}. Hence, to study the convergence of the spectrum of \eqref{eqz} it is important to know the spectrum of the limit problem, therefore we are lead to consider the following problem 
\begin{equation}\label{eqzl}
\begin{cases}
\ds -z'' - \dfrac{N-1}{r}z'-p_\alpha C_{N,\alpha}{\la^{2+\al}r^\alpha \over (1+\la^{2+\al} r^{2+\al})^2}z=\Lambda {z\over r^2}\quad \text{in}\ (0,+\infty) \\
z \in \mathcal{E}.
\end{cases}
\end{equation}
where
$$\mathcal{E}=\Big\{\psi\in C^1[0,+\infty)\ \Big| \ \int_0^{+\infty} r^{N-1}|\psi'(r)|^2\,dr<+\infty\Big\}.$$
It is known that the problem \eqref{eqzl} has an incresing sequence of eigenvalues, which we denote by 
$$\Lambda_1(\al) < \Lambda_2(\al) < \dots < \Lambda_j(\al)< \dots \quad j \in\N$$ 
where
\begin{equation}\label{l12}
\Lambda_1(\al)=-\frac{(\al+2)(2N+\al-2)}{4} \quad \text{and} \quad \Lambda_2(\al)=0,
\end{equation}
see \cite[proof of Theorem 1.3]{GGN}. Let us mention that the eigenspace of \eqref{eqzl} associated to the first eigenvalue $\Lambda_1(\al)$ is spanned by the function
\begin{equation}\label{z1a}
z(r)={\la^{2+\al\over 2} r^{2+\al\over2}\over (1+\la^{2+\al}r^{2+\al})^{N+\al\over 2+\al}}.
\end{equation}
Similarly to \cite{GGN} we shall obtain that \eqref{eqmuk} has a solution $\al_k^\ve$ only for $j=1$. Hence, we will be interested in the first eigenvalue $\Lambda_{1}^{\varepsilon}(\al)$ and its asymptotic behavior as $\ve\to0$.

\medskip
\begin{lemma} \label{lem6}
Let $\Lambda_{1}^{\varepsilon}(\al)$ and $z_{1,\ve,\al}$ denote respectively the first eigenvalue and the positive eigenfunction of \eqref{eqz}, where $z_{1,\ve,\al}$ is normalized with $\norm {z_{1,\ve,\al}}_{L^\infty}=1$. 
Given a compact interval $K \subset (0,+\infty)$, there exists $\ve_0>0$ and a constant $C>0$ independent of $\ve$ and $\alpha$ such that for all $0<\ve<\ve_0$ and $\alpha \in K$ we have the uniform decay
\begin{equation}\label{estzzp}
\abs{z_{1,\ve,\al}'(r)} \leqslant {C\over r^{N-1}},\qquad \abs{z_{1,\ve,\al}(r)}\leqslant {C\over r^{N-2}}.
\end{equation}
Moreover,
\begin{equation}\label{convz}
z_{1,\ve,\al}(r)\to z_{1,\alpha}(r) \quad \text{as} \  \ve\to 0,
\end{equation} 
uniformly for $r \in [0,+\infty), \ \alpha \in K$, where $z_{1,\al}$ is the positive eigenfunction of \eqref{eqzl} with $\|z_{1,\al}\|_{L^\infty}=1$  and
\begin{equation*}  \label{convla}
\Lambda_{1}^{\varepsilon}(\al)\to \Lambda_{1}(\al) \quad \text{as} \  \ve\to 0,
\end{equation*}
uniformly for $\al \in K$.

\end{lemma}


\begin{proof}
Arguing as in \cite[Lemmas 3.3 and 3.4]{GGN}, we first observe that
\begin{equation} \label{dl1}
\limsup_{\ve \to 0} \ \Lambda_{1}^{\varepsilon}(\al) < 0
\end{equation}
uniformly for $\al \in K$. Indeed, by standard computations, taking into account \eqref{erem5}, one can prove \eqref{dl1} using the variational charaterization of $\Lambda_1^\ve(\al)$ with the test function  
$$\psi(r)={\la^{2+\al\over 2} r^{2+\al\over2}\over (1+\la^{2+\al}r^{2+\al})^{N+\al\over 2+\al}}\phi_\ve(r),$$ 
where $\phi_\ve\in C_0^\infty[0,\rho_\ve)$ is a suitable cut-off function. 

Let $z_{1,\ve,\al}$ be the first positive eigenfunction of \eqref{eqz} related to $\Lambda_{1}^{\varepsilon}(\al)$ and normalized with respect to $L^\infty$-norm, namely, $z_{1,\ve,\al}$ satisfies
\begin{equation}\label{eqzn}
\begin{cases}
\ds -z'' - \dfrac{N-1}{r}z'-\tilde w_{\ve,\al} z=\Lambda_{1}^{\varepsilon}(\al) {z\over r^2}\quad \text{in}\ (0,\rho_\ve) \\
z(0) = z(\rho_\ve)=0,\quad \norm{z}_\infty=1,
\end{cases}
\end{equation}
where for simplicity we denote $\tilde w_{\ve,\al}(r)=(p_\alpha-\ve) C_{N,\alpha}r^\alpha w_{\ve,\alpha}^{p_\alpha-1-\ve}(r)$. By using \eqref{dw}, for $\ve$ small enough and $\al \in K$ we have that for some constant $C>0$ independent of $\ve$ and $\al$ it holds
\begin{equation} \label{d26}
\tilde w_{\ve,\al}(r)\leqslant {C r^{\al}\over (1+r^{2+\al})^{2-{N-2\over 2+\al}\ve}}.
\end{equation}
By \eqref{dl1} and the previous inequality, there exist $\varepsilon_0>0$ depending only on $K$ and some $r_0>0$ independent of $\ve$ and $\al\in K$ such that 
\begin{equation} \label{d27}
r^{N-1}\tilde w_{\ve,\al}(r)+\Lambda_{1}^{\varepsilon}(\al)r^{N-3}<0, \quad \text{for all } \ r > r_0, \ \alpha \in K \text{ and } 0< \varepsilon < \varepsilon_0.
\end{equation}
Multiplying \eqref{eqzn} by $r^{N-1}$ and integrating on $(r,\rho_\ve)$ we get that
\begin{equation}\label{i1}
r^{N-1}z_{1,\ve,\al}'(r)=\rho_\ve^{N-1}z_{1,\ve,\al}'(\rho_\ve)+\int_r^{\rho_\ve} \lf[s^{N-1}\tilde w_{\ve,\al}(s)+\Lambda_{1}^{\varepsilon}(\al)s^{N-3}\rg]z_{1,\ve,\al}(s)\,ds,
\end{equation}
and hence, we find that
\begin{equation}\label{ddzn}
z'_{1,\ve,\al}(r)<0\qquad\text{ for }r > r_0
\end{equation}
in view of \eqref{d27}, $0\le z_{1,\ve,\al}(r)\le 1$ and $z_{1,\ve,\al}'(\rho_\ve)<0$. Now, multiplying \eqref{eqzn} by $r^{N-1}$, integrating on $(0,r)$ and using \eqref{d27} and \eqref{dl1} we get that for $r > r_0$

\begin{align}
-r^{N-1}z_{1,\ve,\al}'(r) &= \int_0^{r} \lf[s^{N-1}\tilde w_{\ve,\al}(s)+\Lambda_{1}^{\varepsilon}(\al)s^{N-3}\rg]z_{1,\ve,\al}(s)\, ds  \notag \\
& \leqslant \int_0^{r_0} \lf[s^{N-1}\tilde w_{\ve,\al}(s)+\Lambda_{1}^{\varepsilon}(\al)s^{N-3}\rg]z_{1,\ve,\al}(s)\, ds \notag \\
& \leqslant \int_0^{r_0} s^{N-1}\tilde w_{\ve,\al}(s)z_{1,\ve,\al}(s)\, ds \label{dss}
\end{align}

From \eqref{ddzn}, the sign of $z_{1,\ve,\al}$, \eqref{d26} and \eqref{dss} we find that
$$\abs{r^{N-1}z_{1,\ve,\al}'(r) }\le C\int_0^{r_0} {s^{N-1+\al}\over (1+ s^{2+\al})^{2-{N-2\over 2+\al}\ve_0}}ds$$
and this shows \eqref{estzzp}.

Now, let us study \eqref{convz}. From \eqref{eqzn} and \eqref{estzzp} we find that
$$\int_0^{+\infty} s^{N-1}\abs{z_{1,\ve,\al}'(s)}^2\,ds\le C,$$
where $z_{1,\ve,\al}$ is assumed to be zero for $s>\rho_\ve$. Thus, for every sequence $\ve_n>0$ there is a subsequence (still denoted by $\ve_n$) such that $z_{1,\ve_n,\al}\to z^*$ weakly in $\mathcal{E}$ for some function $z^*\in\mathcal{E}$, hence a.e. in $(0,+\infty)$ and uniformly on compact subsets of $[0,+\infty)$. Using \eqref{estzzp} again, we can pass to the limit into \eqref{eqzn} getting that $z^*\neq 0$ since, from \eqref{estzzp} the maximum point of $z_{1,\ve_n,\al}$ converges to a point $\bar r\in [0,+\infty)$ and $|z^*(\bar r)|=1$ from the uniform convergence. Hence, $z^*$ is a solution of \eqref{eqzl} corresponding to the eigenvalue $\Lambda_1(\al)$ with $\|z^*\|_{L^\infty}=1$. Therefore, $z^*=z_{1,\al}$.  Moreover, from the uniform decay \eqref{estzzp} the convergence $z_{1,\ve,\al}\to z_{1,\al}$ is uniform on the whole $[0,+\infty)$. The convergence $z_{1,\ve,\al}\to z_{1,\al}$ is uniform for $\al\in K$, because otherwise there are sequences $\ve_n>0$, $\ve_n\to 0$ and $\al_n\in K$ such that $z_{1,\ve_n,\al_n}$ are uniformly far from $z_{1,\al_n}$. Arguing as in Lemma \ref{lem3} we get a contradiction since from \eqref{estzzp} there is a convergent subsequence and we can pass to the limit into \eqref{eqzn}.

Finally, we prove the uniform convergence of $\Lambda_{1}^{\varepsilon}(\al)$ to $\Lambda_1(\al)$ for $\al\in K$. Multiply \eqref{eqzn} by $r^{N-1}z_{1,\al}(r)$ and integrate on $(0,\rho_\ve)$, and also multiply \eqref{eqzl} by $r^{N-1}z_{1,\ve,\al}(r)$ and integrate on $(0,\rho_\ve)$. Then, we subtract getting
\begin{equation*}
\begin{split}
-\rho_\ve^{N-1} z_{1,\ve,\al}'(\rho_\ve)z_{1,\al}(\rho_\ve)=&\,\int_0^{\rho_\ve}\lf[s^{N-1}\tilde w_{\ve,\al}(s)-p_\al C_{N,\al}{\la^{2+\al} s^{N-1+\al}\over (1+\la^{2+\al} s^{2+\al})^2}\rg]z_{1,\ve,\al}(s)z_{1,\al}(s)\,ds\\
&\,+ [\Lambda_{1}^{\varepsilon}(\al)-\Lambda_1(\al)] \int_0^{\rho_\ve}s^{N-3}z_{1,\ve,\al}(s)z_{1,\al}(s)\,ds.
\end{split}
\end{equation*}
On one hand, from \eqref{estzzp} and the definition of $z_{1,\al}$ we find that
$$\rho_\ve^{N-1} z_{1,\ve,\al}'(\rho_\ve)z_{1,\al}(\rho_\ve)=o_\ve(1)$$
for $\ve$ small enough, uniformly in $\al\in K$, while
$$\int_0^{\rho_\ve}\lf[s^{N-1}\tilde w_{\ve,\al}(s)-p_\al C_{N,\al}{\la^{2+\al}s^{N-1+\al}\over (1+\la^{2+\al}s^{2+\al})^2}\rg]z_{1,\ve,\al}(s)z_{1,\al}(s)\,ds=o_\ve(1)$$
as $\ve\to0$, uniformly for $\al\in K$, in view of
$$\tilde w_{\ve,\al}(s)\to p_\al C_{N,\al}{\la^{2+\al} s^{\al}\over (1+\la^{2+\al}s^{2+\al})^2}$$
as $\ve\to0$, uniformly for $\al\in K$, \eqref{estzzp} and dominated convergence. On the other hand, we have that
$$\int_0^{\rho_\ve}s^{N-3}z_{1,\ve,\al}(s)z_{1,\al}(s)\,ds\to \int_0^{+\infty}s^{N-3}z_{1,\al}^2(s)\,ds > 0$$
as $\ve\to 0$, uniformly for $\al\in K$. Therefore
$$\sup_{\al\in K}|\Lambda_{1}^{\varepsilon}(\al)-\Lambda_1(\al)|\to 0 \quad \text{ as } \ve \to 0$$
and this concludes the proof.
\end{proof}

\medskip
\begin{remark} \label{rem1}
Using the implicit function Theorem we can prove that $\alpha \mapsto \Lambda_{1}^{\varepsilon}(\al)$ is $C^1$. Actually one can prove that $\Lambda_{1}^{\varepsilon}(\al)$ is analytic, c.f. \cite[Proof of Lemma 2.2, part (c)]{DW} see also \cite{AG}.
\end{remark}

\begin{remark}\label{alke}
Notice that $\Lambda_1(\al)=\ds -{\al^2\over 4}-{\al N\over 2} + 1 - N$, so that, $\Lambda_1$ is strictly decreasing for $\al\geqslant 0$. Furthermore, the equation $\Lambda_1(\al)=-\sigma_k$ with $\al\geqslant0$ is satisfied only for $\al_k=2(k-1)$. Hence, by the uniform convergence of $\Lambda_1^\ve \to \Lambda_1$ as $\ve\to 0$ in the compact set $K_\delta=[\al_k-\delta,\al_k+\delta]$ for any $\delta >0$, it follows that there is $\ve_0>0$ (possibly depending on $\delta$) such that for all $0<\ve \leqslant \ve_0$ we have that $\Lambda_1^\ve(\al_k-\delta)> -\sigma_k>\Lambda_1^\ve(\al_k+\delta)$ (with fixed $\delta$). By continuity of $\Lambda_1^\ve$, there exists at least one $\al_k^\ve\in K_\delta$ such that 
\begin{equation}\label{eq74}
\Lambda_1^\ve(\al_k^\ve)=-\sigma_k,\quad k\geqslant 1.
\end{equation}

\end{remark}

\bigskip
Now, we show that equation \eqref{eqmuk} does not have any solution if $j\geqslant 2$. This fact clearly follows from the following result choosing any $\eta< N-1$. 

\begin{lemma} \label{lem7}
Let $\Lambda_2^\ve(\al)$ denote the second eigenvalue of \eqref{eqz}. Then, given $\eta>0$ and a compact interval $K\subset [0,+\infty)$, there exists $\ve_0>0$ such that
\begin{equation}\label{l2e}
\Lambda_2^\ve(\al)\geqslant -\eta,
\end{equation}
for all $0<\ve<\ve_0$ and $\al\in K$.
\end{lemma}

\begin{proof}
Suppose, by contradiction that \eqref{l2e} does not hold. Then, there exist sequences $\ve_n>0$, $\al_n\in K$ and numbers $\al^*\in K$, $\Lambda^*<0$ such that $\ve_n\to 0$, $\al_n\to \al^*$ and $\Lambda_n:=\Lambda_2^{\ve_n}(\al_n)\to \Lambda^*$ as $n\to+\infty$. Let $z_{2,n}$ be an eigenfunction of \eqref{eqz} associated to $\Lambda_n$ satisfying $\|z_{2,n}\|_{L^\infty}=1$. Consider also $z_{1,n}$ the positive eigenfunction of \eqref{eqz} normalized with $\|z_{1,n}\|_{L^\infty}=1$. We recall that $z_{1,n}$ is positive, $z_{2,n}$ changes sign once on the interval $(0,\rho_{\ve_n})$ and they are orthogonal in the sense of 
\begin{equation}\label{z1z2o}
\int_0^{\rho_{\ve_n}} s^{N-3} z_{1,n}(s) z_{2,n}(s)\,ds=0\ \ \text{for all }n.
\end{equation}
By Lemma \ref{lem6}, $z_{1,n}$ has the uniform decay \eqref{estzzp} and $z_{1,n}\to z_{1,\al^*}$, where $z_{1,\al^*}$ is the positive eigenfunction of \eqref{eqzl} with $\|z_{1,\al^*}\|_{L^\infty}=1$. We shall use that, in this case for $n$ sufficiently large, $z_{2,n}$ also has the same decay as $z_{1,n}$, namely,
\begin{equation}\label{dz2n}
\abs{z_{2,n}'(r)}\leqslant {C\over r^{N-1}}\ \ \ \text{and}\ \ \
\abs{z_{2,n}(r)}\leqslant {C\over r^{N-2}}
\end{equation}
for some constant $C>0$ independent of $n$. Let us assume \eqref{dz2n} for a moment. Hence, as in Lemma \ref{lem6}, we get that $z_{2,n}\to z^*$ weakly in $\mathcal{E}$, a. e. in $(0,+ \infty)$ and uniformly $[0,+\infty)$. Since $\|z_{2,n}\|_{L^\infty}=1$ and \eqref{dz2n}, the function $z^*\ne 0$. Moreover, $z^*$ satisfies the limit equation \eqref{eqzl} for $\al=\al^*$ and $\Lambda=\Lambda^*<0$. Since $\Lambda^*$ is a negative eigenvalue of equation \eqref{eqzl}, it must be the first eigenvalue and \eqref{l12} implies that $\Lambda^*=\Lambda_1(\al^*)$. Therefore, we have that the limit function $z^*$ is an eigenfunction of \eqref{eqzl} associated to its first eigenvalue, so $z^*$ is a nonzero multiple of $z_{1,\al^*}$. Now, taking into account the estimates \eqref{estzzp} and \eqref{dz2n} we can pass to the limit in \eqref{z1z2o} to conclude that
$$\int_0^{+\infty } s^{N-3} z_{1,\al^*}^2(s) \,ds=0,$$
which is a contradiction since $z_{1,\al^*}$ is a positive function. 

To finish the proof it remains to show \eqref{dz2n}. Indeed,  as in \eqref{d27}, by \eqref{d26} and the limit $\Lambda_n \to \Lambda^* < 0$, we have that for $n$ large enough there exists some $r_0$ independent of $n$ such that
\begin{equation}\label{d27n}
r^{N-1}\tilde w_n(r)+ \Lambda_n r^{N-3}<0,\ \ \text{for all }\ r>r_0.
\end{equation}
where $\tilde w_n = \tilde w_{\ve_n,\alpha_n}$ as defined in \eqref{eqzn}. As in \eqref{i1}, we get that
\begin{equation}\label{i2}
r^{N-1}z_{2,n}'(r)=\rho_{\ve_n}^{N-1}z_{2,n}'(\rho_{\ve_n})+\int_r^{\rho_{\ve_n} } \lf[s^{N-1}\tilde w_{n}(s)+\Lambda_n s^{N-3}\rg]z_{2,n}(s)\,ds,
\end{equation}
for all $r\in [0,\rho_{\ve_n})$ and for all $n\geqslant 1$. Let $r_n$ be the point where $z_{2,n}$ changes sign. From \eqref{i2} with $r=r_n$ we obtain
\begin{equation}\label{i2n}
r_n^{N-1}z_{2,n}'(r_n)=\rho_{\ve_n}^{N-1}z_{2,n}'(\rho_{\ve_n})+\int_{r_n}^{\rho_{\ve_n} } \lf[s^{N-1}\tilde w_{n}(s)+\Lambda_n s^{N-3}\rg]z_{2,n}(s)\,ds.
\end{equation}
Without loss of generality, we consider that 
\begin{equation}
\label{d2n}
z_{2,n}>0 \text{ in } (0,r_n) \text{ and } z_{2,n}<0 \text{ in } (r_n,\rho_{\ve_n}).
\end{equation}

We claim that the sequence $\{r_n\}$ remains bounded, since otherwise we would find that the right hand side of \eqref{i2n} is positive for $n$ sufficiently large because of \eqref{d27n}, \eqref{d2n}  and $z_{2,n}'(\rho_{\ve_n})>0$, while the left hand side is negative because of $z_{2,n}'(r_n)<0$, so we get a contradiction.

Now, we can assume that $r_n\leqslant r_0$ for all $n\geqslant 1$. By \eqref{d27n} and the sign of $z_{2,n}$ in $(r_0,\rho_{\ve_n})$ we get that $z_{2,n}'(r)>0$ for all $r\geqslant r_0$. As in \eqref{dss}, taking into account \eqref{d27n} and the signs of $z_{2,n}$ and $z_{2,n}'$, we obtain that for $r> r_0$
\begin{align}
r^{N-1}|z_{2,n}'(r)|  = r^{N-1}z_{2,n}'(r) &= - \int_0^{r} \lf[s^{N-1}\tilde w_{n}(s)+\Lambda_n s^{N-3}\rg]z_{2,n}(s)\, ds \notag \\
& \leqslant  - \int_0^{r_0} \lf[s^{N-1}\tilde w_{n}(s)+\Lambda_n s^{N-3}\rg]z_{2,n}(s)\, ds \notag \\
& \leqslant C. \notag
\end{align}
Therefore, we get the estimate \eqref{dz2n} and this finishes the proof.
\end{proof}

It is important to notice that, for small fixed $\varepsilon > 0$, the values of $\alpha$ such that $w_{\varepsilon,\alpha}$ is degenerate (i.e., $\Lambda_1^{\varepsilon}(\alpha) = -\sigma_k$) are isolated. We can deduce this property using that $\Lambda_1^{\varepsilon}$ is analytic with respect to $\alpha$, see Remark \ref{rem1}.


\section{Proof of the main result}\label{proofresult}

Now we are in position to prove our main theorem. We shall use some ideas as in \cite{GGPS} and \cite{GGN}. We shall prove the theorem for problem \eqref{eq1} by proving the bifurcation result for problem \eqref{eq2}. Actually, we consider the problem
\begin{equation} \label{eq8}
\begin{cases}
-\Delta w = C_{N,\alpha}\abs{x}^{\alpha} |w|^{p_{\alpha}-\varepsilon-1}w, & \ \ \text{in} \ B_{\rho_{\varepsilon}}\\
w=0 & \ \ \text{on} \ \partial B_{\rho_{\varepsilon}}
\end{cases}
\end{equation}
and show the existence of branches of nonradial solutions of \eqref{eq8} bifurcating from $(\alpha_k^{\varepsilon},w_{\varepsilon,\alpha_k^{\varepsilon}})$, then we prove that these solutions are positive and consequently solve \eqref{eq2}. 

First we need some notations. Consider the set $\mathcal{S}(\varepsilon)$ defined by
\begin{equation*}\label{S}
\mathcal{S}(\varepsilon):=
\left\{
\begin{split}
\,\,(\alpha,w_{\varepsilon,\alpha})\in  (0,+\infty)\times C^{1,\gamma}_0(\overline B_{\rho_{\varepsilon}})\, \hbox{ such that } \\ w_{\varepsilon,\alpha}
     \hbox{ is the unique radial}
\hbox{ solution of (\ref{eq2})} \,\,\,
\end{split}
\right\}
\end{equation*}
and recall that, given the curve $\mathcal{S}(\varepsilon)$,
a point $(\alpha^{\varepsilon},w_{\varepsilon,\alpha^{\varepsilon}})\in \mathcal{S}(\varepsilon)$
is a {\em  nonradial bifurcation point } if in  every neighborhood of
$(\alpha^{\varepsilon},w_{\varepsilon,\alpha^{\varepsilon}})$  in $(0,+\infty)\times C^{1,\gamma}_0(\overline B_{\rho_{\varepsilon}})$
there exists a point $(\alpha,v_{\varepsilon,\alpha})$ such that $v_{\varepsilon,\alpha}$ is a nonradial solution of \eqref{eq8}.

\begin{proof}[Proof of Theorem \ref{t1}] $\phantom{X}$ \\

\noindent {\bf Step 1.} Since the solution $w_{\varepsilon, \alpha}$ is always radially nondegenerate, a necessary condition for bifurcation is that the linearized equation \eqref{plw} admit a nontrivial solution. By the arguments in section \ref{linprob}, it is equivalent to study the eigenvalue problem \eqref{eqz} and the equation \eqref{eqmuk}. We shall prove the existence of a bifurcation point using the change in the Morse index of the solution $w_{\varepsilon, \alpha}$, by Lemma \ref{lem3} this solution approaches the solution $U_{\alpha}$ of the limit problem \eqref{problemrn} as $\varepsilon$ tends to $0$ and in \cite{GGN} it was proved that the Morse index of $U_{\alpha}$ changes as the parameter $\alpha$ crosses an even integer. 

Let $\alpha_k = 2(k-1)$ with $k \in \N$, $k \geqslant 2$ and consider the compact interval $[\alpha_k - \rho, \alpha_k + \rho]$, notice that without loss of generality we can assume that \linebreak $0 < \rho < 1$. By Lemma \ref{lem6}, Remark \ref{alke} and Lemma \ref{lem7} there exists $\varepsilon_0 > 0$ such that

\begin{equation} \label{eq73a}
\begin{split}
& \Lambda_1^{\varepsilon}(\alpha_k-\rho) > - \sigma_k > \Lambda_1^{\varepsilon}(\alpha_k+\rho) \\
& \Lambda_2^{\varepsilon}(\alpha) > - \sigma_1, \quad \forall \ 
\alpha \in [\alpha_k - \rho, \alpha_k + \rho]
\end{split}
\end{equation}
for all $0 < \varepsilon < \varepsilon_0$, where $\sigma_k$ was defined in \eqref{autflapb} and \eqref{autvlapb}. Hence, by continuity there exists at least one solution to the equation 

\begin{equation} \label{eq74a}
\Lambda_1^{\varepsilon}(\alpha) = - \sigma_k, \qquad \alpha \in [\alpha_k - \rho, \alpha_k + \rho].
\end{equation}

Moreover, we can take $\varepsilon_0$ smaller, if necessary, so that there are no other solutions of $ \Lambda_1^{\varepsilon}(\alpha) = - \sigma_l$
in $[\alpha_k - \rho, \alpha_k + \rho]$, with $l \neq k$, for all $ 0 < \varepsilon < \varepsilon_0$.

From now on, we fix $\varepsilon \in (0 , \varepsilon_0)$. 
Since the function $\Lambda_1^{\varepsilon}(\alpha)$ in analytic (see Remark \ref{rem1}), it follows that the solutions of \eqref{eq74a} are isolated and hence, by \eqref{eq73a}, we can conclude that there exists some $\alpha_k^{\varepsilon} \in (\alpha_k - \rho, \alpha_k + \rho)$ and $0 < \delta \leqslant \rho$ such that

\begin{equation} \label{bfpt}
\begin{split}
&{(i)}\ \ \  \alpha_k^{\varepsilon} \text{ is the unique solution of } 
\Lambda_1^{\varepsilon}(\alpha) = - \sigma_k \text{ in } [\alpha_k^{\varepsilon} - \delta, \alpha_k^{\varepsilon} + \delta]; \\
&{(ii)} \ \ \Lambda_1^{\varepsilon}(\alpha) \neq - \sigma_l,
 \text{ for all } l \neq k \text{ and } \alpha \in [\alpha_k^{\varepsilon} - \delta, \alpha_k^{\varepsilon} + \delta]; \\
&{(iii)} \   \Lambda_1^{\varepsilon}(\alpha) > - \sigma_k \text{ for all } \alpha_k^{\varepsilon} - \delta \leqslant \alpha  < \alpha_k^{\varepsilon};\\
&{(iv)} \ \ \Lambda_1^{\varepsilon}(\alpha) < - \sigma_k \text{ for all } \alpha_k^{\varepsilon}< \alpha  \leqslant \alpha_k^{\varepsilon} + \delta.
\end{split}
\end{equation}

In this case, for $\alpha = \alpha_k^{\varepsilon}$ there exists a nontrivial solution $v_k$ of \eqref{eqvk}, which implies that all nontrivial solutions of the linearized equation \eqref{plw} have the form
$$ v(x) = v_k(|x|)Y_k\left( \frac{x}{|x|} \right)$$
with $Y_k \in \text{Ker}(\lap_{\mathbb{S}^{N-1}}+\sigma_k)$. Arguing as in \cite[proof of Theorem 1.3]{GGN2} one can prove that, when $\alpha$ crosses $\alpha_k^{\varepsilon}$, the Morse index of $w_{\varepsilon, \alpha}$ increases by the dimension of the eigenspace $\text{Ker}(\lap_{\mathbb{S}^{N-1}}+\sigma_k)$.

\noindent {\bf Step 2.} Now we prove that the point $\alpha_k^{\varepsilon}$ defined as in the previous step gives rise to a bifurcation point. We restrict our
attention to the subspace $\mathcal{H}_{\varepsilon}$ of $C^{1,\gamma}_0(\overline B_{\rho_{\varepsilon}})$ given by
\begin{equation*}\label{2.3}
\mathcal{H}_{\varepsilon}:=\left\{
v\in C^{1,\gamma}_0(\overline B_{\rho_{\varepsilon}}) \, ,\,
\hbox{s.t. }v(x_1,\dots,x_N)=v(g(x_1,\dots, x_{N-1}),x_N)\,\atop \hbox{
  for any } g\in O(N-1)
\right\}
\end{equation*}
where $0 < \gamma < 1$ and $O(N-1)$ is the orthogonal group in  $\R^{N-1}$.

By a result of Smoller and Wasserman in
\cite{SW}, we have that for any $k$ the eigenspace  $V_k$ of the
Laplace-Beltrami operator on  $\mathbb{S}^{N-1}$, spanned by the
eigenfunctions corresponding to the
eigenvalue  $\sigma_k$ which are $O(N-1)$ invariant is
one-dimensional. Hence, the Morse index $m(\alpha)$ of the radial solution $w_{\varepsilon,\alpha}$ in
the space $\mathcal{H}_{\varepsilon}$ satisfies 
\begin{equation}\label{2.4}
m(\alpha_k^{\varepsilon}+\delta)-m(\alpha_k^{\varepsilon}-\delta)=1
\end{equation} 

Let us consider the operator $T_{\varepsilon}(\alpha,v):(0,+\infty)\times \mathcal{H}_{\varepsilon}\to \mathcal{H}_{\varepsilon}$,
defined by  $T_{\varepsilon}(\alpha,v):=\left(-\Delta\right)^{-1}\left( C_{N,\alpha}|x|^{\alpha}|v|^{p_{\alpha}-\varepsilon-1}v\right)$. Notice that  $T_{\varepsilon}$  is a compact operator
for every fixed $\alpha$ and is continuous with respect to $\al$. Let us suppose, by contradiction, that $(\al_k^{\varepsilon},w_{\varepsilon,\al_k^{\varepsilon}})$ is not a
bifurcation point and set
$F_{\varepsilon}(\al,v):=v-T_{\varepsilon}(\al,v)$. In this case, there exists $\delta>0$ small enough such that we have \eqref{bfpt}, \eqref{2.4} and  
\begin{equation}\label{2.5}
F_{\varepsilon}(\alpha,v)\neq 0,\quad \forall \,\, (\alpha,v) \in [\alpha_k^{\varepsilon}-\delta,\alpha_k^{\varepsilon}+\delta] \times \left( \overline{B}_{\delta}(w_{\varepsilon,\alpha}) \backslash \{ w_{\varepsilon,\alpha} \} \right),
\end{equation}
where $\overline{B}_{\delta}(w_{\varepsilon,\alpha})$ is the closed ball in $\mathcal{H}_{\varepsilon}$ centered in $w_{\varepsilon,\alpha}$ with radius $\delta$.

Let us consider the
set  $\Gamma:=\{(\alpha,v)\in[\alpha_k^{\varepsilon}-\delta,\alpha_k^{\varepsilon}+\delta]\times  \mathcal{H}_{\varepsilon}\,:\, \norm{v- w_{\varepsilon,\alpha}}_{\mathcal{H}_{\varepsilon}}
<\delta \}$.  Notice that
$F_{\varepsilon}(\alpha,\cdot)$ is a compact perturbation of the identity and so it makes
sense to consider the Leray-Schauder topological degree $\mathit{deg}
\left(  F_{\varepsilon}(\alpha,\cdot),\Gamma_{\alpha},0\right)$ of $F_{\varepsilon}(\alpha,\cdot)$ on the set
 $\Gamma_{\alpha}:=\{v\in \mathcal{H}_{\varepsilon}\hbox{ such that } (\alpha,v)\in \Gamma\}$. From
\eqref{2.5} it follows $F_{\varepsilon}$ is an admissible homotopy, therefore by the homotopy invariance of
the degree, we get
\begin{equation}\label{2.6}
\mathit{deg} \left( F_{\varepsilon}(\alpha,\cdot),\Gamma_{\alpha},0\right)\hbox{ is constant on }[\alpha_k^{\varepsilon}-\delta,\alpha_k^{\varepsilon}+\delta].
\end{equation}
Since the linearized  operator $(F_{\varepsilon})_{v}(\alpha,w_{\varepsilon,\alpha})$ is invertible for  $\alpha=\alpha_k^{\varepsilon}+\delta$ and
$\alpha=\alpha_k^{\varepsilon}-\delta$, we have
$$\mathit{deg} \left( F_{\varepsilon}(\alpha_k^{\varepsilon}-\delta,\cdot),\Gamma_{\alpha_k^{\varepsilon}-\delta},0\right)=(-1)^{m(\alpha_k^{\varepsilon}-\delta)}$$
and
$$\mathit{deg} \left( F_{\varepsilon}(\alpha_k^{\varepsilon}+\delta,\cdot),\Gamma_{\alpha_k^{\varepsilon}+\delta},0\right)=(-1)^{m(\alpha_k^{\varepsilon}+\delta)}.$$

By  \eqref{2.4} it follows that
$$\mathit{deg} \left( F_{\varepsilon}(\alpha_k^{\varepsilon}-\delta,\cdot),\Gamma_{\alpha_k^{\varepsilon}-\delta},0\right)=-\mathit{deg} \left(
F_{\varepsilon}(\alpha_k^{\varepsilon}+\delta,\cdot),\Gamma_{\alpha_k^{\varepsilon}+\delta},0\right)$$
contradicting \eqref{2.6}. Then $(\alpha_k^{\varepsilon},w_{\varepsilon,\alpha_k^{\varepsilon}})$ is a bifurcation point of \eqref{eq8} 
and the bifurcating solutions are nonradial since $w_{\varepsilon,\alpha} $ is radially
nondegenerate for any $\alpha$ as proved in Lemma \ref{lem4}.

\noindent {\bf Step 3.} Finally, we prove that the branch of nonradial solutions of \eqref{eq8} which bifurcate from $w_{\varepsilon,\alpha_k^{\varepsilon}}$ contains only positive solutions and therefore it is a branch of nonradial solutions of \eqref{eq2}.

Let $\Sigma_{\varepsilon}$ denote the set
\begin{equation*}
\Sigma_{\varepsilon} := \overline{\{ (\alpha,v)\in \R_{+} \times \mathcal{H}_{\varepsilon} \ | \ F_{\varepsilon}(\al,v)=0 \ \text{and} \ v \neq w_{\varepsilon,\alpha}\}}
\end{equation*}
i.e. the closure, in $\R_{+} \times \mathcal{H}_{\varepsilon}$, of the set of solutions of $F_{\varepsilon}(\al,v)=0$ different from $w_{\varepsilon,\alpha}$. Notice that, by the arguments in the previous step, the pair $(\alpha_k^{\varepsilon},w_{\varepsilon,\alpha_k^ {\varepsilon}})$ is a nonradial bifurcation point and then belongs to $\Sigma_{\varepsilon}$. For $(\alpha_k^{\varepsilon},w_{\varepsilon,\alpha_k^ {\varepsilon}}) \in \Sigma_{\varepsilon}$ let $\mathcal{C}(\alpha_k^{\varepsilon}) \subset \Sigma_{\varepsilon}$ denote the closed connected component of $\Sigma_{\varepsilon}$ which contains $(\alpha_k^{\varepsilon},w_{\varepsilon,\alpha_k^ {\varepsilon}})$ and it is maximal with respect to the inclusion.

Consider the set $\mathcal{C} \subset \mathcal{C}(\alpha_k^{\varepsilon})$ of points $(\alpha, w_{\alpha})$ for which $w_{\alpha}$ is a positive solution. Since  $(\alpha_k^{\varepsilon},w_{\varepsilon,\alpha_k^ {\varepsilon}}) \in \mathcal{C}$ then $\mathcal{C} \neq \emptyset$.
By standard arguments (see \cite[Theorem 3.3, Step 1]{G}) one can prove that the set $\mathcal{C}$ is open and closed in $\mathcal{C}(\alpha_k^{\varepsilon})$ and therefore it is a branch of nonradial solutions of \eqref{eq2}.

Since we proved the existence of nonradial solutions of \eqref{eq2}, by inverting the transformation \eqref{trans} we get the existence of nonradial solutions of \eqref{eq1}. This finishes the proof.
\end{proof}

\begin{center}
{\bf Acknowledgements}
\end{center}

\noindent The first author was partially supported by grant Fondecyt Iniciaci\'on 11130517, Chile.

\end{document}